\documentclass[12pt, a4paper]{amsart}

\usepackage{accents,array,graphicx,hyperref}

\newtheorem{theorem}{Theorem}

\newtheorem{cor}{Corollary}

\newcommand{\mcC}{\mathcal{C}}
\newcommand{\mcB}{\mathcal{B}}

\theoremstyle{definition}
\newtheorem{defn}{Definition}[section]

\theoremstyle{plain}
\newtheorem{thm}{Theorem}[section]
\newtheorem{lem}[thm]{Lemma}
\newtheorem{prop}[thm]{Proposition}

\title[TQC with MCG Representations]
{On Topological Quantum Computing with Mapping Class Group Representations}
\author{Wade Bloomquist}
\email{bloomquist@math.ucsb.edu}
\address{Dept. of Mathematics\\
    University of California\\
    Santa Barbara, CA 93106-6105\\
    U.S.A.}

\author{Zhenghan Wang}
\email{zhenghwa@microsoft.com, zhenghwa@math.ucsb.edu}
\address{Microsoft Station Q and Dept. of Mathematics\\
 University of California\\
   Santa Barbara, CA 93106-6105\\
    U.S.A.}

\thanks{The second author thanks M. Barkeshli for asking him related questions, and is partially supported by NSF grants DMS-1411212 and  FRG-1664351.}


\begin{document}

\begin{abstract}
We propose an encoding for topological quantum computation utilizing quantum representations of mapping class groups.  Leakage into a non-computational subspace seems to be unavoidable for universality. We are interested in the possible gate sets which can emerge in this setting.  As a first step, we prove that for abelian anyons, all gates from these mapping class group representations are normalizer gates.  Results of Van den Nest then allow us to conclude that for abelian anyons this quantum computing scheme can be simulated efficiently on a classical computer.  With an eye toward more general anyon models we additionally show that for Fibonnaci anyons, quantum representations of mapping class groups give rise to gates which are not generalized Clifford gates.
\end{abstract}

\maketitle

\section{Introduction}

Experimental topological quantum computation hinges on a trade-off between the computational power of anyons and their detection and control in laboratories.  So far none of the experimentally accessible anyons are univeral through braiding alone like the Fibonacci anyon \cite{FLW02,RW}.  Therefore, designing protocols to supplement these braidings remains interesting in the search for a universal gate set.  In \cite{BF}, it is shown that all mapping class group representations arising from an abelian anyon model are, in principle, accessible for quantum computation.  It is known that all quantum representations of mapping class groups coming from abelian anyons have finite image, and hence cannot be used to construct a universal gate set \cite{FF,Gus}.  In this paper, we prove a stronger result that they actually are all generalized Clifford gates with respect to a natural encoding.  This allows for the application of results on the ability to simulate this computing scheme classically.

Given a unitary modular category or anyon model $\mcB$ of rank $d$, a natural construction gives projective unitary representations of the mapping class group of each oriented closed surface $\Sigma_g$ of genus $g$.  We explore a computational subspace inside of $V(\Sigma_g)$, namely a copy of $V(\Sigma_1)^{\otimes g}$. Intuitively this subspace can be thought of in terms of the handlebodies that these surfaces bound.  The genus $g$ handlebody bounds a trivalent graph which contains $g$ cycles, and these cycles are the subgraphs which correspond to the copy of the torus in our subspace.  With the appropriate choice of this spine, as seen below in Figure \ref{gbasis}, we see that there is a subgroup of $\mathrm{MCG}(\Sigma_g)$ which acts on each component $V(T^2)$ just as $\mathrm{MCG}(T^2)\cong \mathrm{SL}(2,\mathbb{Z})$ does.   This gives us $g$ qudits in $V(\Sigma_g)$, and the image of mapping classes under the quantum representation provides gates on these qudits.

Results of 
Ng and Schauenburg, \cite{NS}, provide a significant limitation as they have shown that the quantum representation of the torus will always have finite image.  This tells us that the $1$ qudit gates form a finite group in this encoding. So if we take our encoding without any leakage, we have no hope of reaching a universal gate set.  In this paper we will turn away from the question of universality, and instead focus on what types of gates can occur in this finite collection.  In the abelian anyon case the computational subspace actually equals the entire space $V(\Sigma_g)$.  Our attention will be on how, in this case, a generalized Knill-Gottesmann theorem can be used to simulate this model classically.

%
\section{Preliminaries}

\subsection{Hilbert space of states}
Given a unitary modular category or anyon model $\mathcal{B}$ with a complete representative set of anyons $\Pi_\mcB$, we construct the Hilbert space of states, or the TQFT Hilbert space, $V(\Sigma_g)$ of any closed oriented surface $\Sigma_g$ of genus $g$ following \cite{RT1,RT2,Tur}.  Embed $\Sigma_g$ in $\mathbb{R}^3$ so that it bounds the standardly embedded handlebody $H_g$ in $\mathbb{R}^3$, then we assign to $\Sigma_g$ a spine, $S$, of $H_g$, i.e. the trivalent graph whose regular neighborhood is $H_g$.  
\begin{defn}
A \textbf{$\Pi_{\mathcal{B}}$ coloring} of a ribbon graph, $\Gamma$, is an assignment of objects in $\mathcal{B}$ to the edges of $\Gamma$ and morphisms of the cyclically ordered colors assigned to the incident half edges to the vertices of $\Gamma$.  
\end{defn}
The Hilbert space $V(\Sigma_g)$ is spanned by the basis of $\Pi_{\mathcal{B}}$-colorings of $S$.  Our choice of $S$, corresponding to one choice of basis for $V(\Sigma_g)$, is shown in Fig. \ref{gbasis}.
\begin{figure}[h!]
\[\includegraphics[scale=.5]{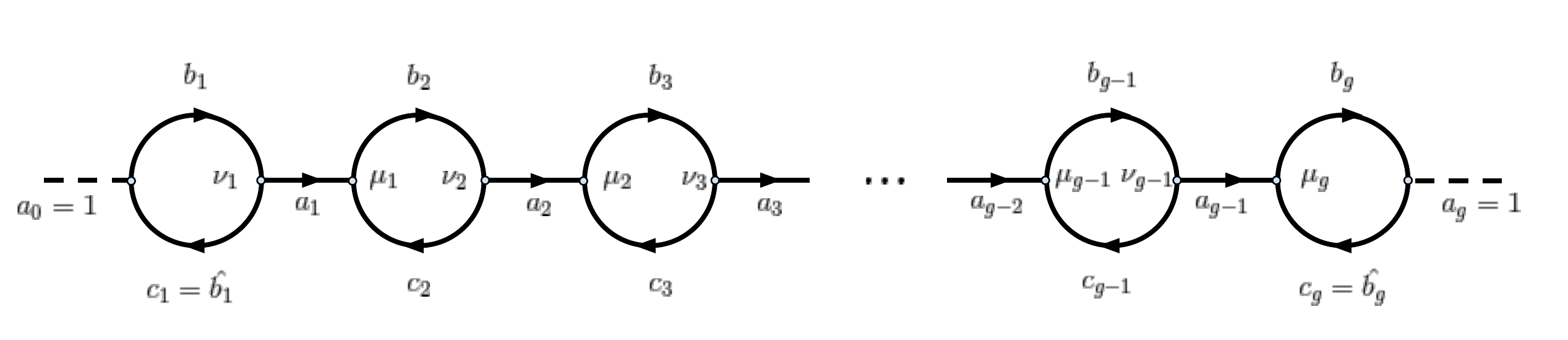}\]
\caption{A basis element of $V(\Sigma_g)$}
\label{gbasis}
\end{figure}
We denote a basis element of this type as 

\[
\vec{v}=(\vec{a},\vec{b},\vec{c},\vec{\mu},\vec{\nu})
\]
\[
=(a_1,..., a_{g-1},b_1,..,b_g,c_1, ...,c_g,\mu_2,...,\mu_g,\nu_1,...,\nu_{g-1})
\]

(with an overall constant from vertices).  That this vector space is actually a (projective) representation of the mapping class group is an immediate consequence of the axioms of a TQFT.  In particular, any mapping class can be described nicely in terms of a cobordism, namely the mapping cyclinder, in such a way that a TQFT must assign a (projective) action on the state space to each mapping class.  A nice overview of this for general $(2+1)$ TQFTs can be found in \cite{Wang}.  

\subsection{The Projective Action}
Let $h:\Sigma\rightarrow \Sigma$ be an orientation preserving homeomorphism in the mapping class group $\mathrm{MCG}(\Sigma)$ of $\Sigma$.  
Consider the cylinder $\Sigma\times [0,1]$ and regard $\Sigma\times\{0\}$ as being parametrized by $id$ and $\Sigma\times \{1\}$ by $h$.  Suppose $H$ is the handlebody bounded by $\Sigma$, with spince $S$ taken as a ribbon graph in $H$ such that the colorings of $S$ are the basis of $V(\Sigma)$.  Then gluing $(H,S)$ to $\Sigma\times\{1\}$ with $h$ and to $\Sigma\times\{0\}$ with the identity results in a pair $(M,\Omega)$ of a closed $3-$manifold $M$ and a ribbon graph $\Omega$ in $M$.  Evaluation of the invariant for this pair gives an operator
\[V_h:V(\Sigma)\rightarrow V(\Sigma).\]
The resulting (projective) representation of $\mathrm{MCG}(\Sigma)$ is called the quantum representation of mapping class groups.  Our focus is on the generators of the mapping class group so $h$ is a positive Dehn twist about a simple closed curve $\gamma$.  Applying the above construction to the specific case of Dehn twists amounts to labeling $\gamma$ with the Kirby color $\omega$, giving $\gamma$ a $-1$ framing relative to the $\Sigma$, and  then evaluating the ribbon graph invariant \cite{Rob}. The Kirby color $\omega$, defined in Figure \ref{Kirb}, is key to this construction and underlies the ability to use surgery to perform these computations diagrammatically.  
\begin{figure}[h!]
\includegraphics[scale=.45]{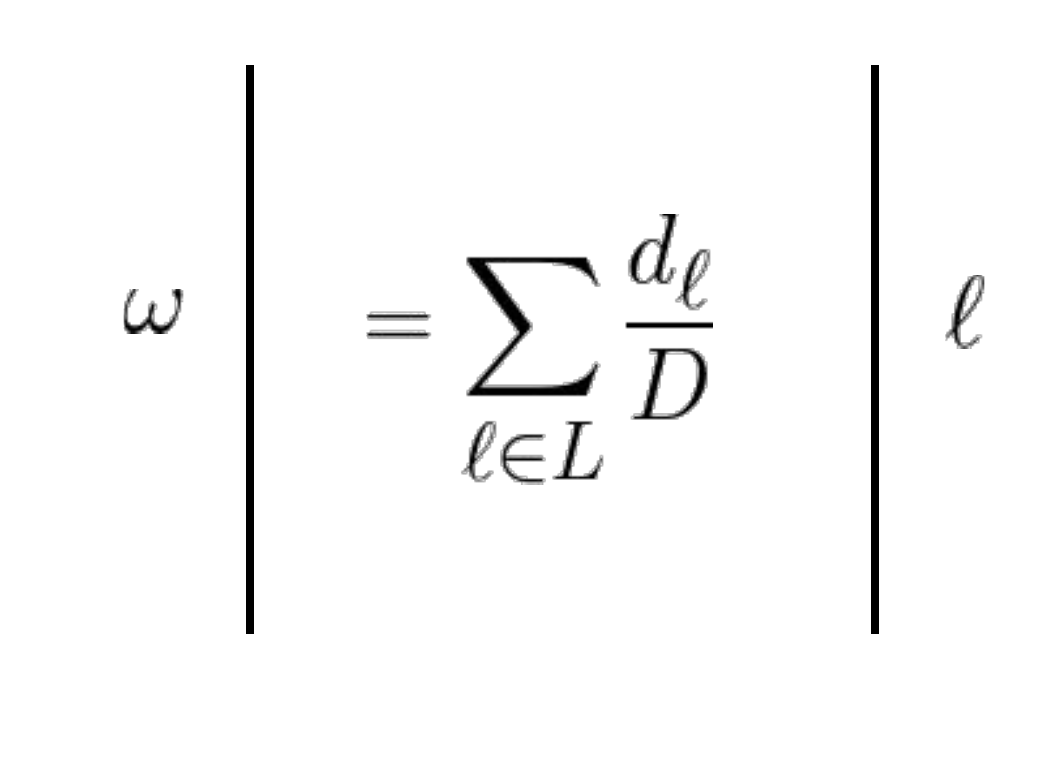}
\caption{}
\label{Kirb}
\end{figure}

\subsection{Abelian Anyon Models}
An abelian anyon model is one in which all quantum dimensions are $1$.  The fusion rules of an abelian anyon model form a finite abelian group, $G$.  A given finite abelian group gives rise to a family of abelian anyon models, which for our purposes should be thought of as indexed by a $3-$cocycle and a pure quadratic form on the group.  We list some of the relevant data for these theories \cite{Sti}:  Let $a,b,c\in G$
\[[F^{a,b,c}_{a+b+c}]_{a+b,b+c}=f(a,b,c)\in H^{3}(G,\mathbb{Q}/\mathbb{Z})\]

\begin{figure}[h!]
\includegraphics[scale=.425]{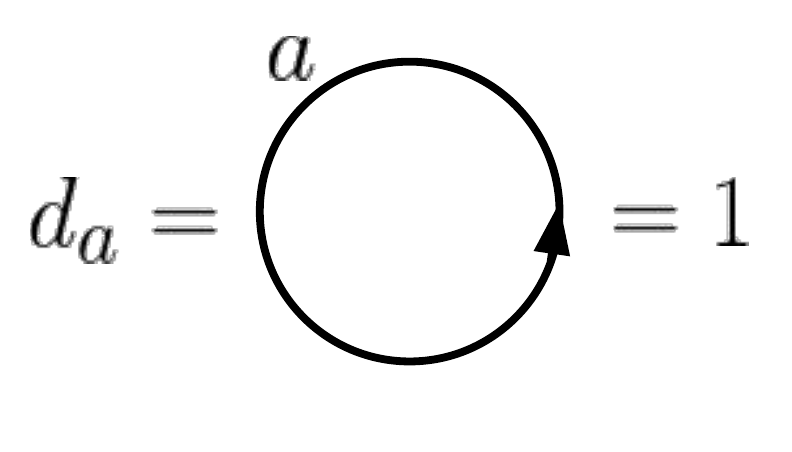}
\end{figure}

\begin{figure}[h!]
\includegraphics[scale=.425]{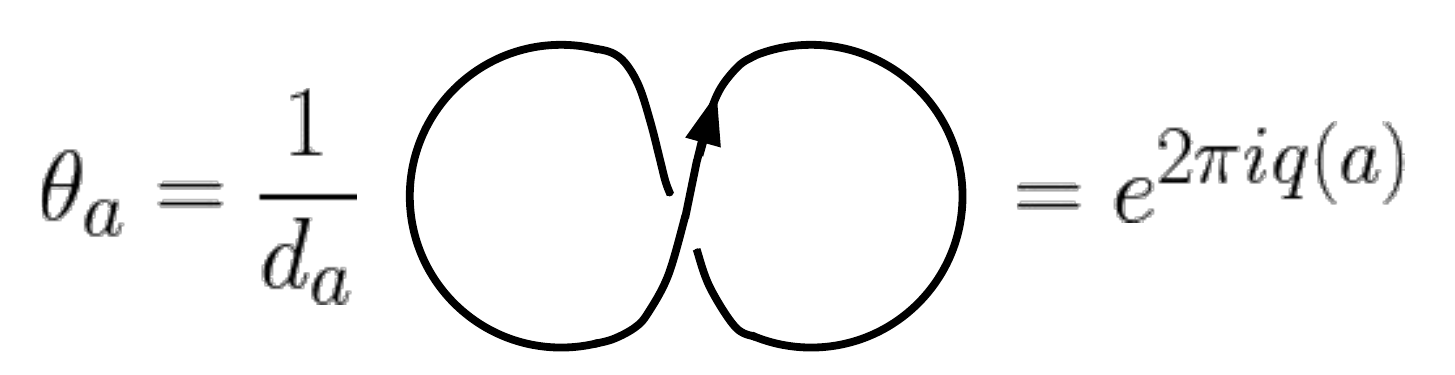}
\end{figure}

Where $q:G\rightarrow \mathbb{R}/\mathbb{Z}$ is a pure quadratic form on $G$.  
\begin{figure}[h!]
\includegraphics[scale=.425]{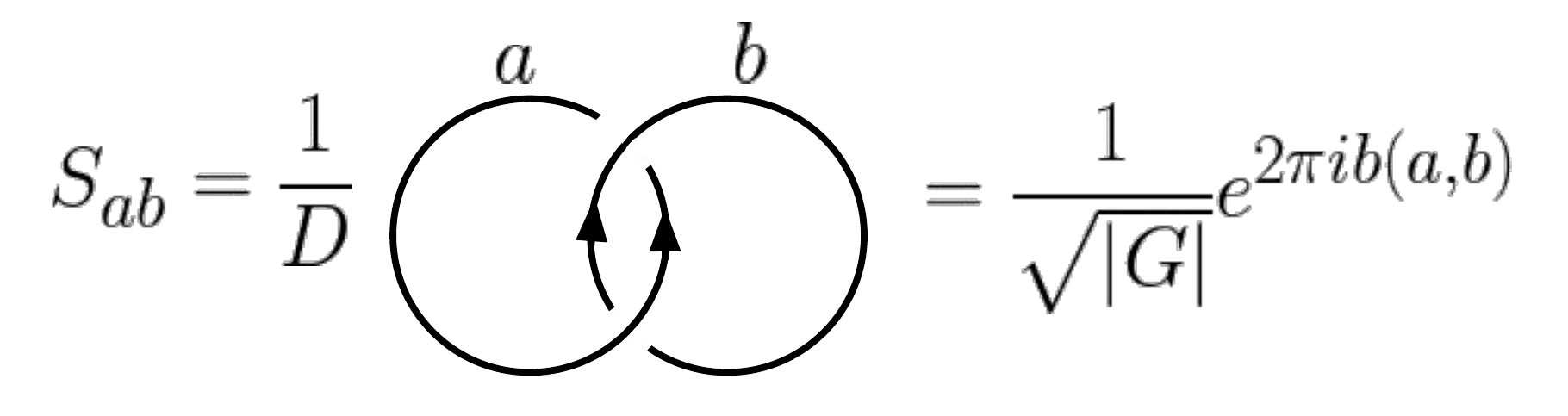}
\end{figure}
\newline
where
\[b(x,y)=q(x+y)-q(x)-q(y)\]
is the bi-linear form associated to $q$.

\subsection{Clifford Groups}
We follow the exposition given in \cite{BLV}.  Let $G$ be a finite abelian group, decomposed as 
\[G=\mathbb{Z}/m_1\mathbb{Z}\times...\times \mathbb{Z}/m_s\mathbb{Z}.\]
We consider the complex vector space
\[\mathcal{H}_G=\mathbb{C}^{m_1}\otimes...\otimes \mathbb{C}^{ m_s},\] 
and so we have
\[\mathcal{H}_G=\mathrm{span}\{|g\rangle:g\in G\}.\]
We note the following observation
\[\mathcal{H}_G^{\otimes n}=\mathcal{H}_{\bigoplus_{i=1}^n G}\]
We will use the notation
\[G^n:=\bigoplus_{i=1}^n G\]

\subsubsection{Pauli Group over $G$}
\begin{defn}
A \textbf{Pauli operator over $G$} is any unitary operator in $U(\mathcal{H}_G)$ of the form
\[\sigma_{(a,g,h)}:=\gamma^aZ_gX_h\]
where
\begin{align*}
&\gamma=e^{\frac{\pi i}{|G|}}\\
&a\in \mathbb{Z}_{2|G|}\\
&X_g(|x\rangle)=|g+x\rangle\\
&Z_h(|x\rangle)=\chi_h(x)|x\rangle \\
\end{align*}
where $\chi_h$ is a character of $G$.
\end{defn}
\begin{defn}
The \textbf{Pauli group over $G$} is the subgroup of $U(\mathcal{H}_G)$ generated by all Pauli operators, denoted $P_{1,G}$.  Then we have
\[P_{n,G}:=(P_{1,n})^{\otimes n}\subseteq U(\mathcal{H}_G^{\otimes n})=U(\mathcal{H}_{G^n})\] 

called the \textbf{$n^{th}$ Pauli Group over $G$}

\end{defn}

\subsubsection{Clifford Group over $G$}
\begin{defn}
The \textbf{$n^{th}$ Clifford Group over $G$}, denoted as $\mcC_{n,G}$, is the normalizer of $P_{n,G}$ in $U(\mathcal{H}_G^{\otimes n})$.  
\end{defn}
As operators differing by only a phase will not contribute to a conjugation we can direct our focus on $PU(\mathcal{H}_G^{\otimes n})$ instead of $PU(\mathcal{H}_G^{\otimes n})$.

\subsubsection{Normalizer Circuits}
\begin{defn}
A \textbf{Normalizer Circuit over $G$} is one composed of the following gate types: 
\begin{itemize}
\item Group automorphism gates:
\[|g\rangle\mapsto |\psi(g)\rangle\]
for $\psi(g)$ a group automorphism.

\item Quadratic phase gates:
\[|g\rangle\mapsto \zeta(g) |g\rangle\]
where $|\zeta(g)|=1$ and 
\[\zeta(g+h)=\zeta(g)\zeta(h)B(g,h)\]
where
\[B(x+y,g)=B(x,g)B(y,g)\]
\[B(g,x+y)=B(g,x)B(y,g).\]

\item Quantum Fourier Transforms:

\[\mathcal{F}: |g\rangle\mapsto \frac{1}{|G|}\sum_{x\in G}\chi_x(g)|x\rangle\]
Where $\chi_x$ are characters of the group.

\end{itemize}
\end{defn}
\begin{thm}\cite{Van}
The subgroup of $U(\mathcal{H}_G)$ generated by normalizer gates over $G^n$ is contained in $C_{n,G}$.  
\end{thm}

\begin{thm}\cite{Van}\label{Vanden}
Any normalizer circuit over any finite abelian group can be classically efficiently simulated in at most polynomial time in the number of quantum Fourier transforms, number of gates in the circuit, the number of cyclic factors of the group, and the logarithm of the orders of the cyclic factors.  
\end{thm}

\section{A General Computation}
\subsection{Mapping Class Group Generators}
We will be working with the Humphries generators of the mapping class group of a genus $g$ surface as seen in Fig. ~\ref{humph}.  
\begin{figure}[h!]
\[\includegraphics[scale=.5]{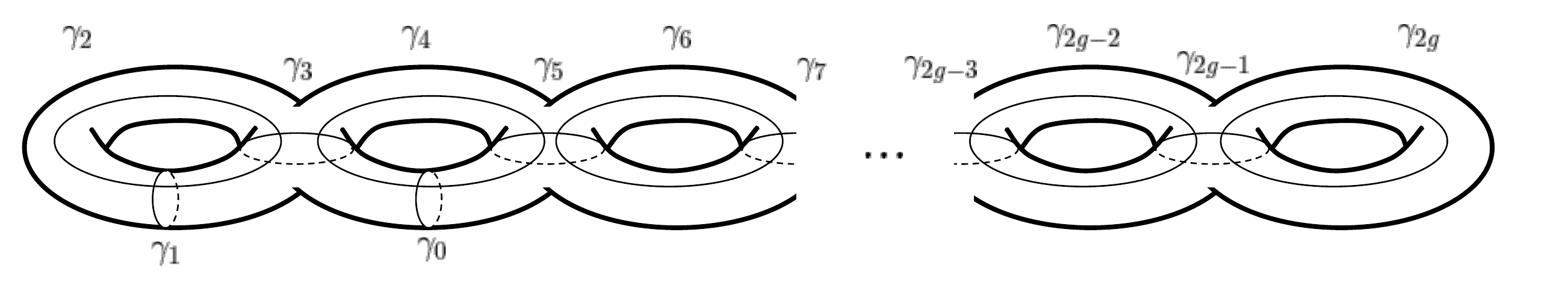}\]
\caption{The Humphries generators}
\label{humph}
\end{figure}
The actual generators of the mapping class group are positive Dehn Twists about these $2g+1$ curves.  In particular we will fix the notation that $T_i$ will stand for the image under the quantum representation of a positive Dehn twist about the curve $\gamma_i$.  
\subsection{$T_0$ and $T_1$}
The local computation seen in Fig. \ref{T0T1} can be applied to the computations of $T_0$ and $T_1$.  
\begin{figure}[h!]
\[\includegraphics[scale=.35]{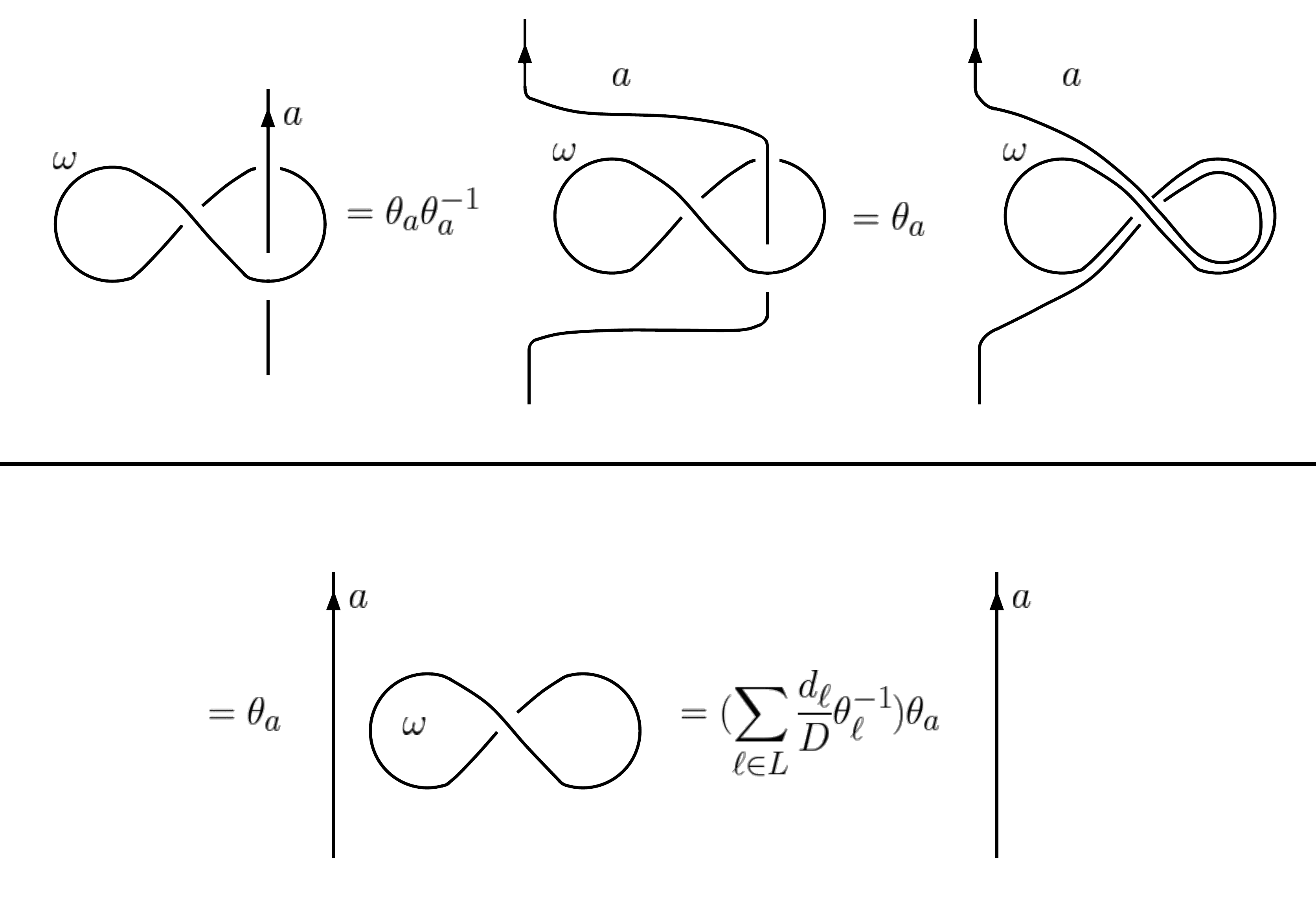}\]
\caption{}
\label{T0T1}
\end{figure}
\newpage
In particular we see that 
\[T_0(\vec{v})=\frac{p_-}{D}\theta_{c_2} \vec{v}\]
and 
\[T_1(\vec{v})=\frac{p_-}{D}\theta_{c_1}\vec{v},\]
where $c_1$ and $c_2$ are the colors on the appropriate edges in $V(\Sigma_g)$ as seen in Figure \ref{gbasis} which are relevant for the Dehn twists about the curves $\gamma_0$ and $\gamma_1$ seen in Figure \ref{humph}, and
\[\frac{p_-}{D}=\sum_{a\in L} \frac{d_a}{D}\theta_a^{-1}=e^{-c\pi i/4}\]
is a root of unity, where $c$ is the central charge \cite{FG}.

\newpage
\subsection{$T_{2i+1}$ for $i=1,...,g-1$}

The local computation shown in Fig. \ref{T2i+11} and Fig. \ref{T2i+12} can be applied to find $T_{2i+1}$ for $i=1,...,g-1$.  
\begin{figure}[h!]
\[\includegraphics[scale=.5]{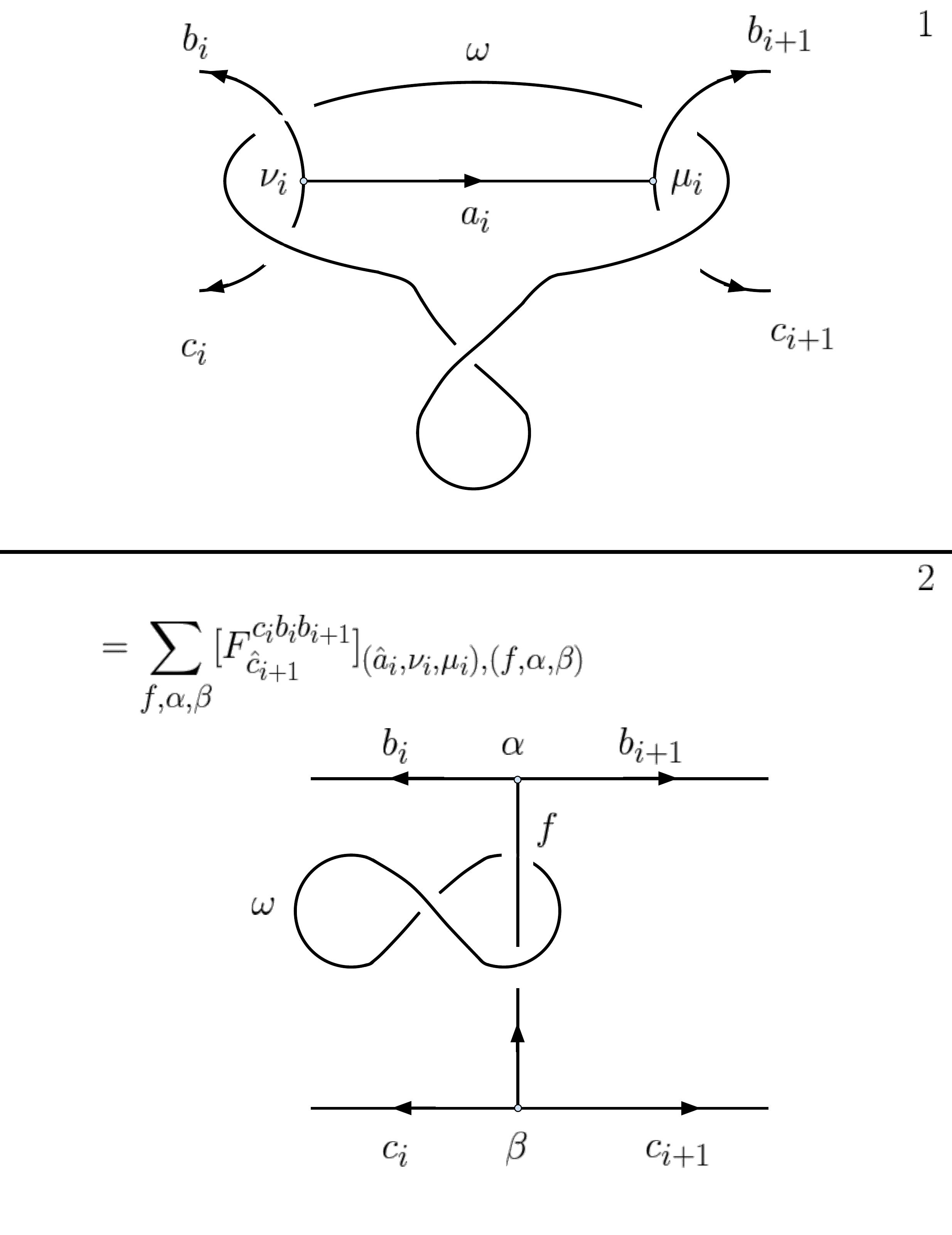}\]
\caption{Steps $1$ and $2$}
\label{T2i+11}
\end{figure}
\begin{figure}[h!]
\[\includegraphics[scale=.5]{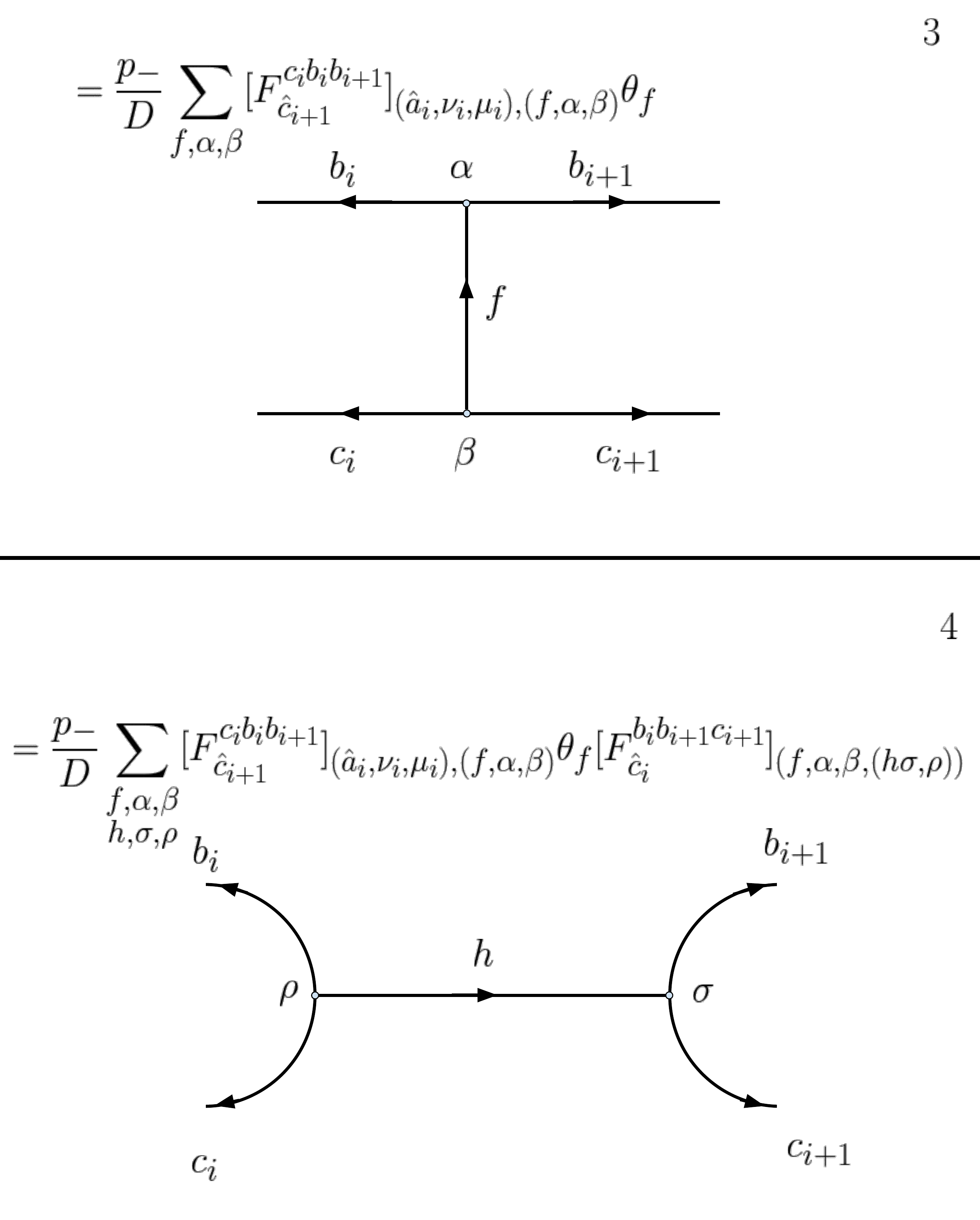}\]
\caption{Steps $3$ and $4$}
\label{T2i+12}
\end{figure}
\newpage
In particular we have 

\[T_{2i+1}(\vec{v})\]
\[=\frac{p_-}{D}\sum_{f,\alpha,\beta,h,\sigma,\rho} [F^{c_ib_ib_{i+1}}_{\hat{c}_{i+1}}]_{(\hat{a}_i,\nu_i,\mu_{i+1}),(f,\alpha,\beta)}\theta_f [F^{b_ib_{i+1}c_{i+1}}_{\hat{c}_{i}}]_{(f,\alpha,\beta),(h,\sigma,\rho)} \vec{v}^\prime_{h,\rho,\sigma},\]
where $\vec{v}^\prime_{h,\rho,\sigma}$ is defined by changing $\vec{v}$ by the following: $a_i$ to $h$, $\nu_i$ to $\rho$, and $\mu_{i+1}$ to $\sigma$.  

\subsection{$T_{2i}$ for $i=1,...,g$}
This computation is much more involved than the previous.  This should be thought of as the generalization of $S-$matrices from the genus $1$ case where the previous examples were more analogous to $T$ matrices.  To begin we look at the evaluation seen in Fig. \ref{ShortCut} which will prove useful in our upcoming computation.  
\begin{figure}[h!]
\[\includegraphics[scale=.35]{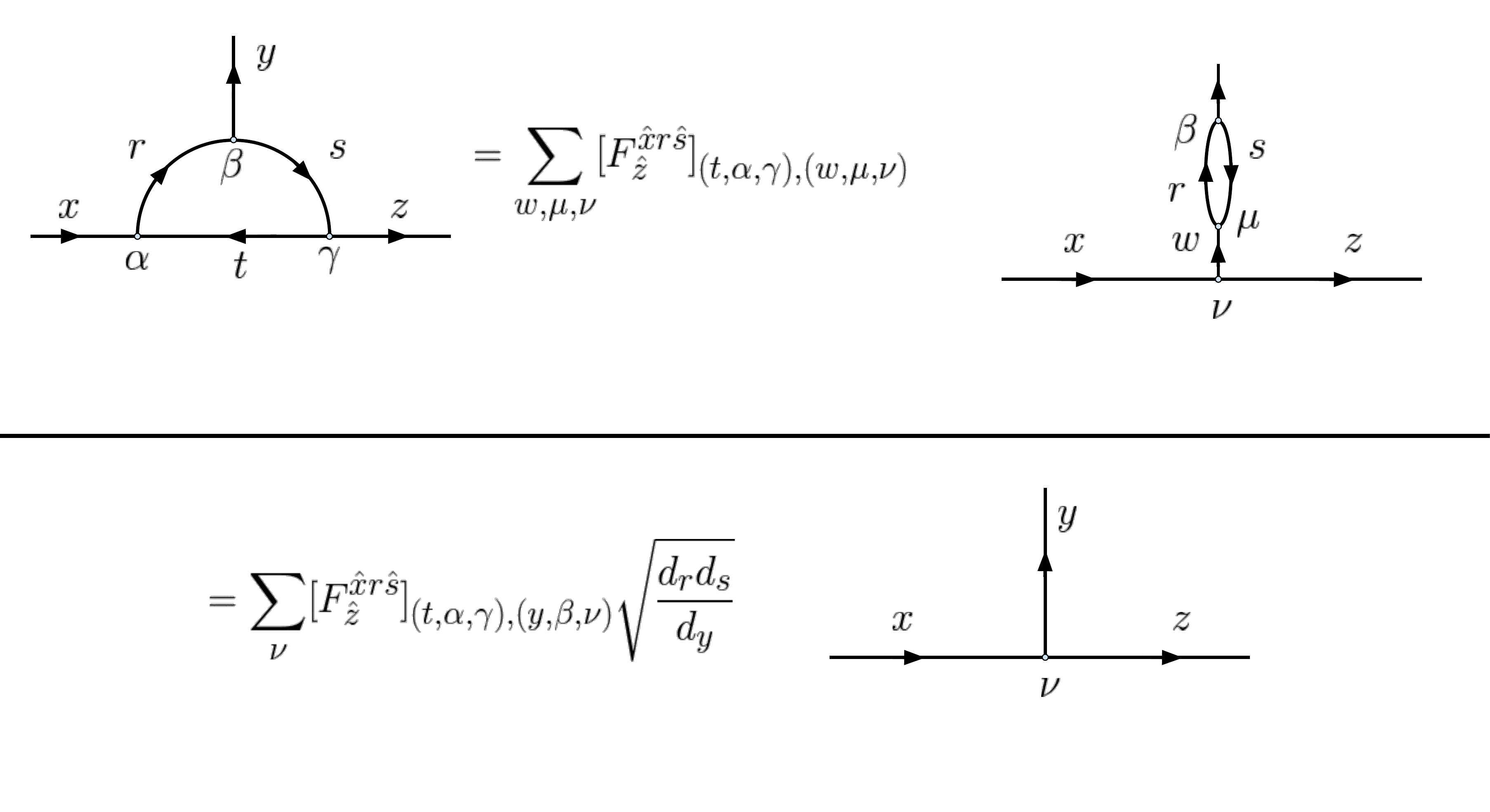}\]
\caption{A useful computation}
\label{ShortCut}
\end{figure}
With this in mind we see in Fig. \ref{Big1} ,Fig. \ref{Big2},Fig. \ref{Big3}, and Fig. \ref{Big4}, a local calculation that allows us to realize the action of $T_{2i}$.
\begin{figure}[h!]
\[\includegraphics[scale=.5]{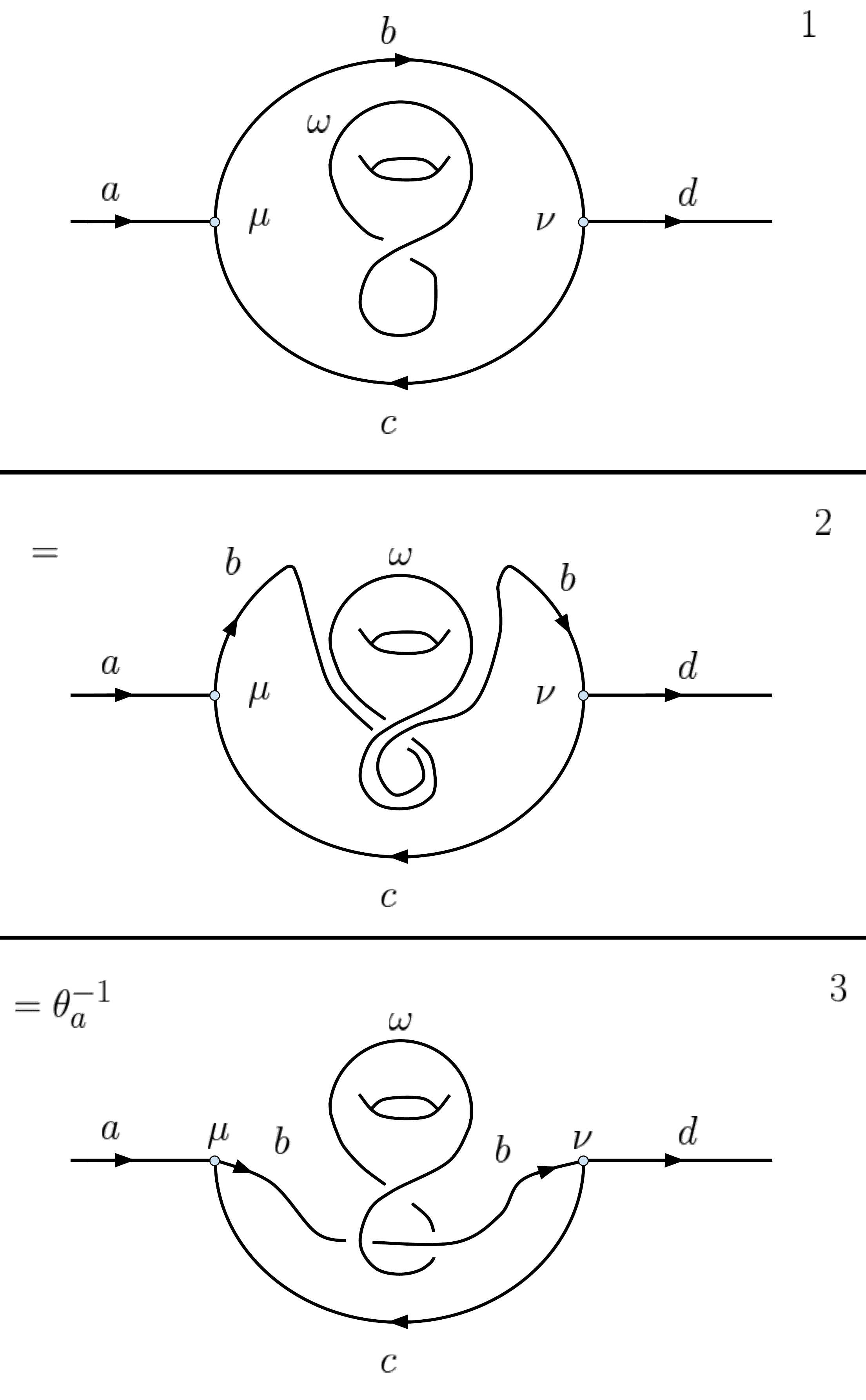}\]
\caption{Steps $1,2,$ and $3$}
\label{Big1}
\end{figure}
\begin{figure}[h!]
\[\includegraphics[scale=.45]{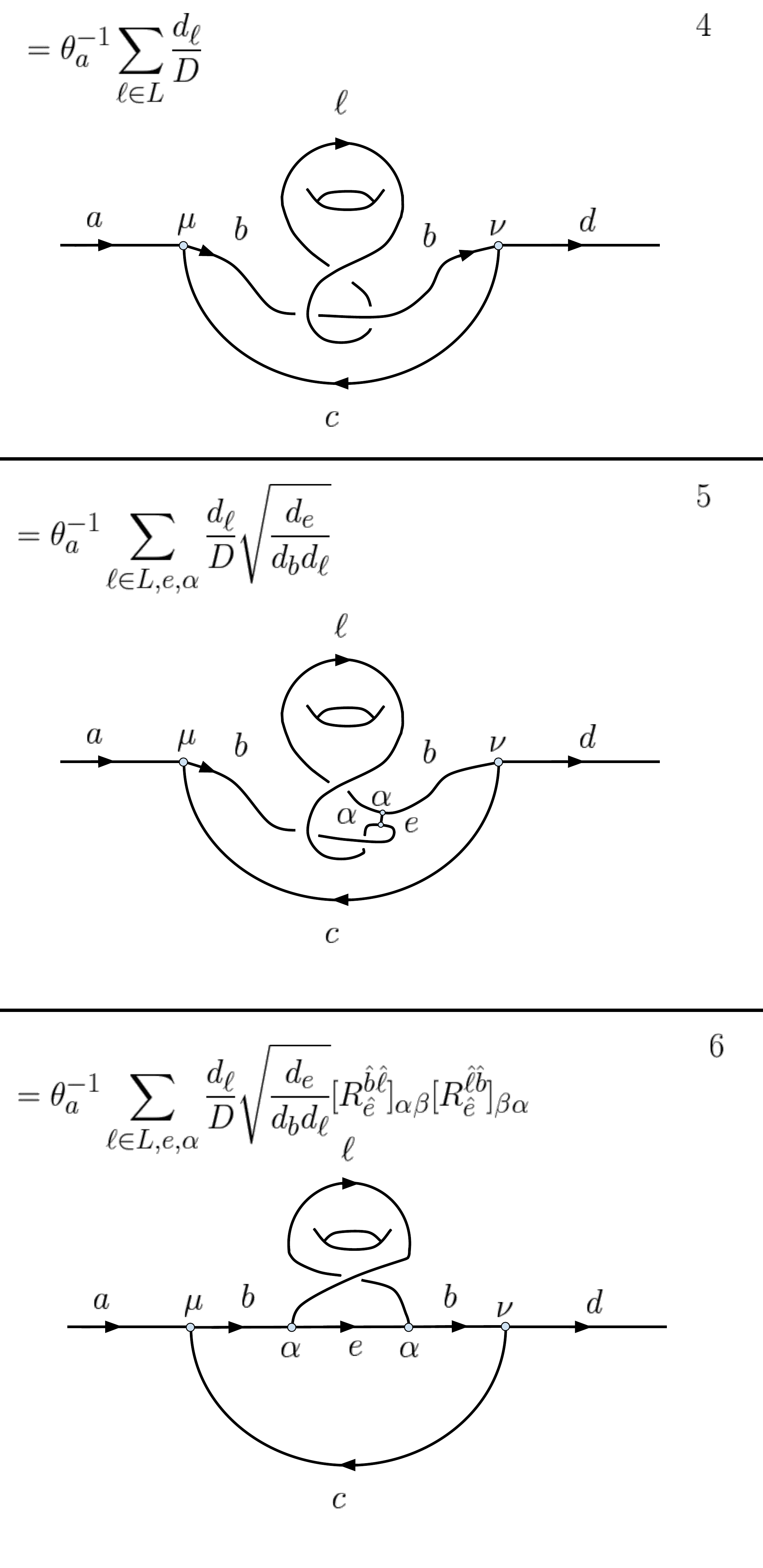}\]
\caption{Steps $4,5,$ and $6$}
\label{Big2}
\end{figure}
\begin{figure}[h!]
\[\includegraphics[scale=.42]{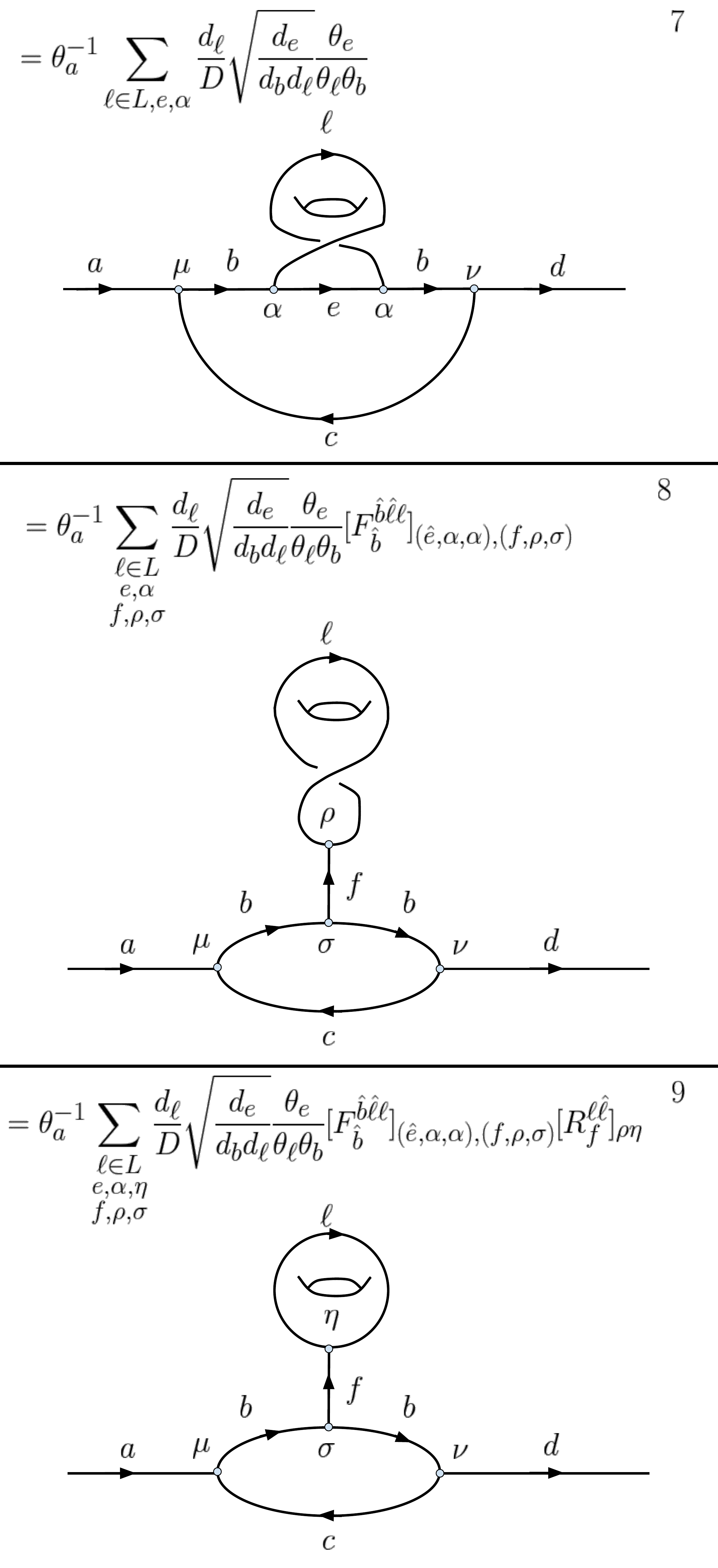}\]
\caption{Steps $7,8,$ and $9$}
\label{Big3}
\end{figure}
\begin{figure}[h!]
\[\includegraphics[scale=.5]{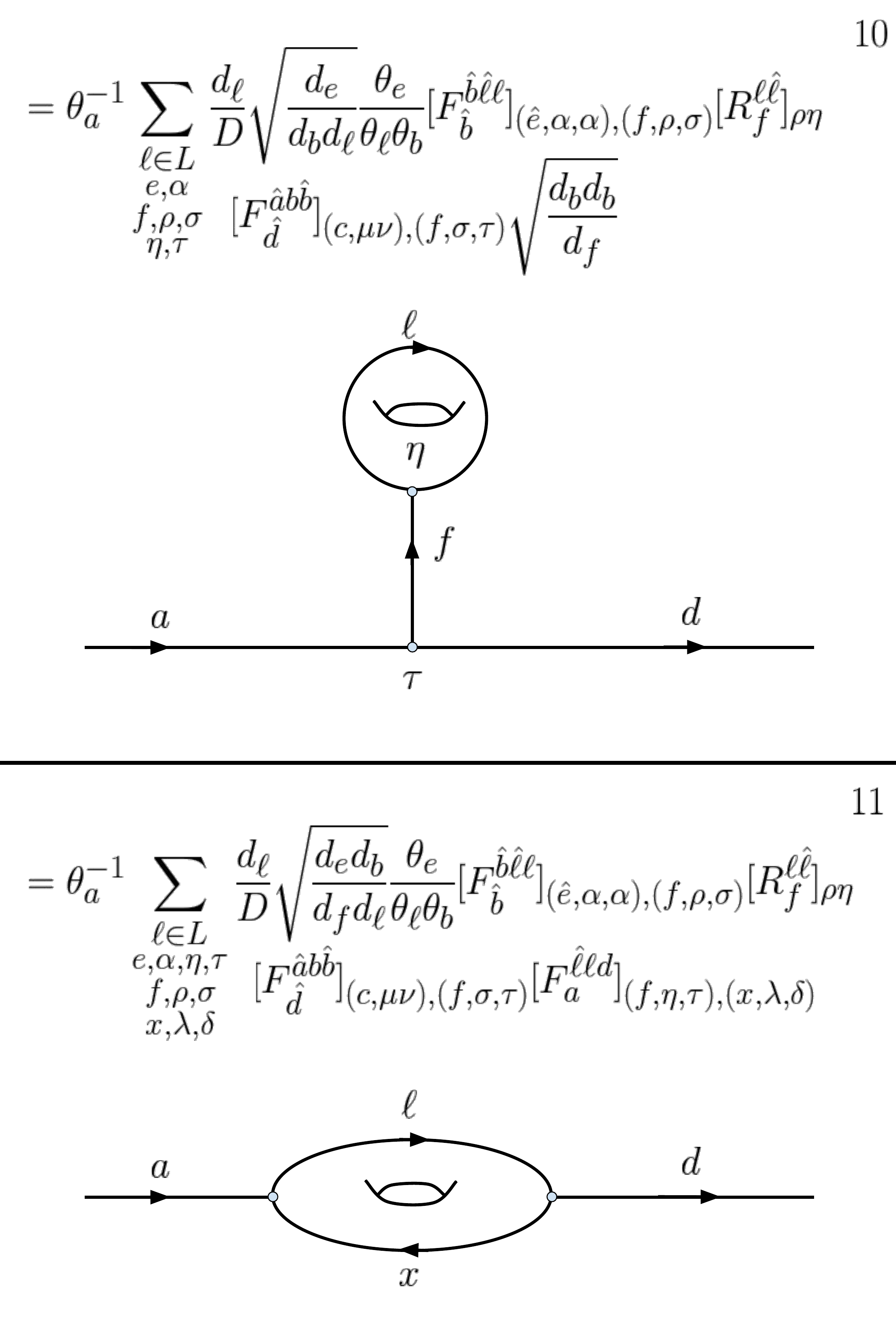}\]
\caption{Steps $10$ and $11$}
\label{Big4}
\end{figure}
\clearpage
Then we see that 
\[T_{2i}(\vec{v})=\theta_{b_i}^{-1}\sum_{\substack{\ell,e,\alpha\\
f,\rho,\sigma,\eta,\tau\\
x,\lambda,\delta}}\frac{d_\ell}{D}\sqrt{\frac{d_ed_{b_i}}{d_{f}d_\ell}}\frac{\theta_e}{\theta_\ell\theta_{b_i}}[F^{\hat{b_i}\hat{\ell}\ell}_{\hat{b_i}}]_{(\hat{e},\alpha,\alpha), (f,\rho,\sigma)}[R^{\ell\hat{\ell}}_f]_{\rho\eta}\]
\[[F^{\hat{a}_{i-1}b_i\hat{b_i}}_{\hat{a_i}}]_{(c_i,\mu_i,\nu_i),(f,\sigma,\tau)}[F^{\hat{\ell}\ell a_i}_{a_{i-1}}]_{(f,\eta,\tau),(x,\lambda,\delta)} \vec{v}^\prime_{\ell,x,\delta,\lambda},\]
where $\vec{v}^\prime_{\ell,x,\delta,\lambda}$ is determined by changing $b_i$ to $\hat{\ell}$, $c_i$ to $\hat{x}$, $\mu_i$ to $\delta$, and $\nu_i$ to $\lambda$.  
\section{Specializing to Abelian Anyon Models}
With these computations in mind we look to ground ourselves with the concrete example discussed in the introduction.  For the remainder of this section let our modular tensor category $\mathcal{C}$ have fusion rules forming a group $G=\mathbb{Z}/m_1\mathbb{Z}\times...\times \mathbb{Z}/m_s\mathbb{Z}$ with $m_i|m_{i+1}$ and modular data determined by $f\in H^3(G,\mathbb{Q}/\mathbb{Z})$ and a pure quadratic form $q$.  
\subsection{Hilbert Spaces of States}
Let $\Sigma_g$ be a closed surface of genus $g$.  We look to describe $V(\Sigma_g)$ concretely.
We note that abelian MTCs are multiplicity free, meaning
\[a\otimes b=a+b=\sum N^{ab}_c c\]
where $N^{ab}_c=\delta_{c,a+b}$, and in particular that the dimension of the $\mathrm{Hom}$ spaces are either $0$ or $1$, meaning we can ignore vertex labels.  Now we also know $\hat{a}=-a$, so we have $a+\hat{a}=0$ and $a+b=0$ exactly when $b=\hat{a}$.  Now we look at the following lemma.
\begin{lem}
When looking at the trivalent graph seen in Fig. \ref{Induct} colored by a finite abelian group $G$.
\begin{figure}[h!]
\[\includegraphics[scale=.4]{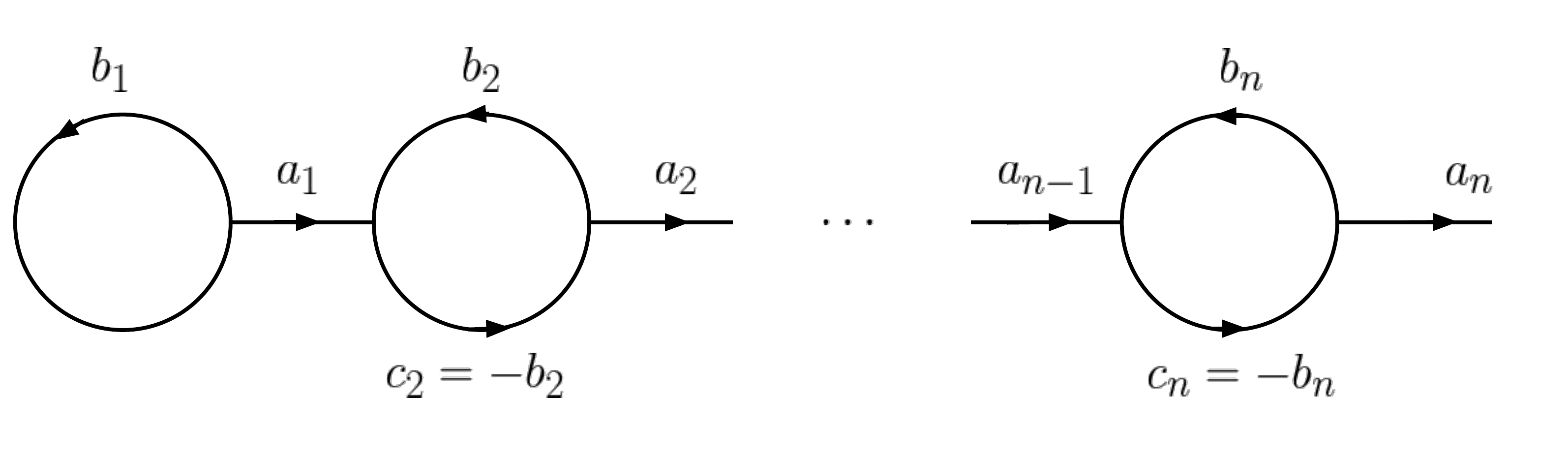}\]
\caption{}
\label{Induct}
\end{figure}
Then $a_i=0$ for all $i$.
\end{lem}
\begin{proof}
This is a simple proof by induction.
\end{proof}
\newpage
Applying this lemma we have:
\begin{prop}
We have $V(\Sigma_g)\cong \mathcal{H}_G^{\otimes g}$, where the basis is given in Fig. \ref{SpecificBasis}.
\begin{figure}[h!]
\[\includegraphics[scale=.45]{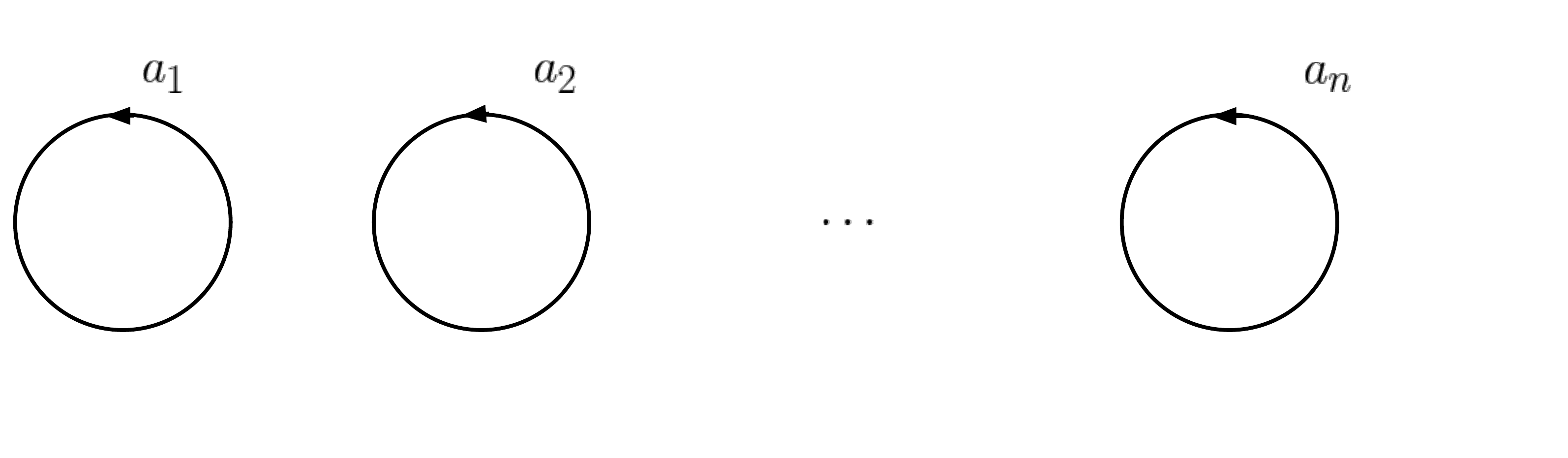}\]
\caption{•}
\label{SpecificBasis}
\end{figure}
\end{prop}
\begin{proof}
This is an immediate application of the above lemma.
\end{proof}
We will denote an element of the basis shown in Fig. \ref{SpecificBasis} as $\vec{a}=(a_1,...,a_g)$.  

\subsection{The $\mathrm{MCG}$ Action}
\subsubsection{$T_0$ and $T_1$}
This computation is identical to that of the general setting.
And so \[T_0(\vec{a})=\frac{p_-}{D}\theta_{a_2}\vec{a}=\frac{p_-}{D}\omega^{a_2^2}\vec{a}\]
and 
 \[T_1(\vec{a})=\frac{p_-}{D}\theta_{a_1}\vec{a}=\frac{p_-}{D}\omega^{a_1^2}\vec{a}.\]
 
Now define
\[L:\mathcal{H}_G\rightarrow\mathcal{H}_G\]
defined by 
\[L(|a\rangle)=\theta_a |a\rangle\]
Then we have
\[T_0=L_2(\vec{a})\]
and
\[T_1=L_1(\vec{a})\]
where 
\[L_i=Id\otimes Id\otimes...\otimes L\otimes Id\otimes...\otimes Id\]
where $L$ acts on the $i^{th}$ component.
\par 
This additional notation is introduced as while $T_0$ and $T_1$ act on $\mathcal{H}_G^{\otimes g}$ they can both be described as acting on a single component $\mathcal{H}_G$.  The above defined $L$ is exactly this map defined abstractly to act on $\mathcal{H}_G$.  
\subsubsection{$T_{2i+1}$ for $i=1,...,g-1$}
This computation is also identical, but we are able to make use of the explicit F-moves.  In particular we have
\[T_{2i+1}(\vec{a})\]
\[=\frac{p_-}{D}\sum_{f,h} [F^{(-a_i)a_i(-a_{i+1})}_{-a_{i+1}}]_{(0,f)}\theta_f [F^{a_i(-a_{i+1})a_{i+1}}_{a_{i}}]_{(f,h)} \vec{a}^\prime_{h},\]
but we have the only non-zero F-move is 
\[[F^{a,b,c}_{a+b+c}]_{a+b,b+c}=f(a,b,c)\in H^3(G,U(1))\]
Now we also note that we elected to describe all F-moves as positive powers, but actually
\[[F^{a_i(-a_{i+1})a_{i+1}}_{a_{i}}]=[F^{(-a_i)a_i(-a_{i+1})}_{-a_{i+1}}]^{-1}\]
\[T_{2i+1}(\vec{a})=f(a_i,-a_{i+1},a_{i+1})\theta_{a_i-a_{i+1}} \overline{f(a_i,-a_{i+1},a_{i+1})}\vec{a}=\theta_{a_i-a_{i+1}}\vec{a}\]
Now define
\[M:\mathcal{H}_G\otimes \mathcal{H}_G\rightarrow \mathcal{H}_G\otimes\mathcal{H}_G,\]
defined by
\[M(|a\rangle\otimes |b\rangle)=\theta_{a-b}|a\rangle\otimes |b\rangle.\]
Observe that $M$ can also be described as follows:
\[M:\mathcal{H}_{G\oplus G}\rightarrow \mathcal{H}_{G\oplus G},\]
where
\[M(|a+b\rangle)=\theta_{a-b}|a+b\rangle.\]

Then we have 
\[T_{2i+1}(\vec{a})=M_{i,i+1}(\vec{a}).\]
\subsubsection{$T_{2i}$ for $i=1,...,g$}
In Fig. \ref{pSpec1} and Fig. \ref{pSpec2} we provide an alternative version of this computation which utilizes many of the specific properties of the fusion rules for abelian MTCs which allow for the use of shortcuts. 
\begin{figure}[h!]
\[\includegraphics[scale=.5]{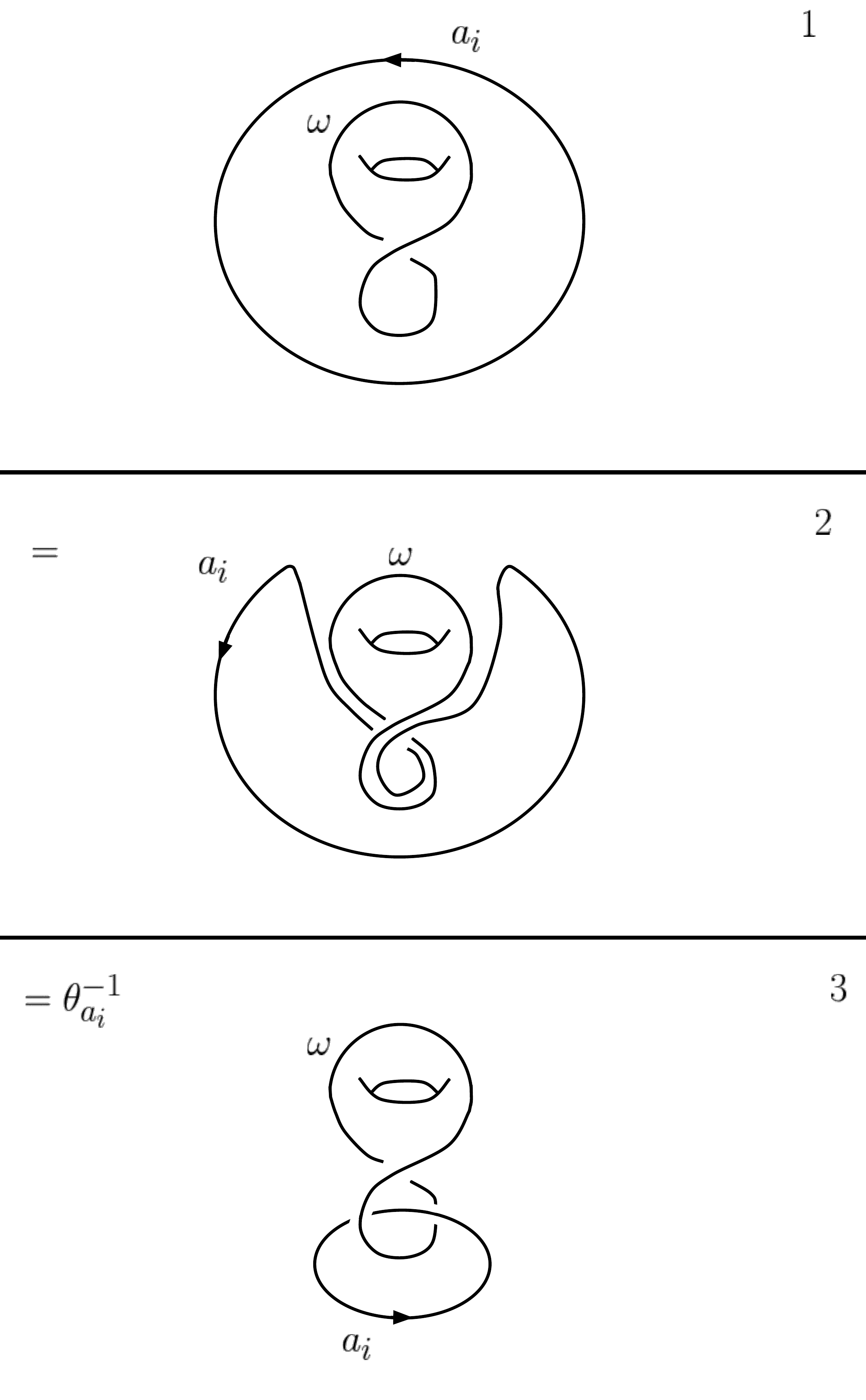}\]
\caption{Steps $1,2,$ and $3$}
\label{pSpec1}
\end{figure} 
\begin{figure}[h!]
\[\includegraphics[scale=.65]{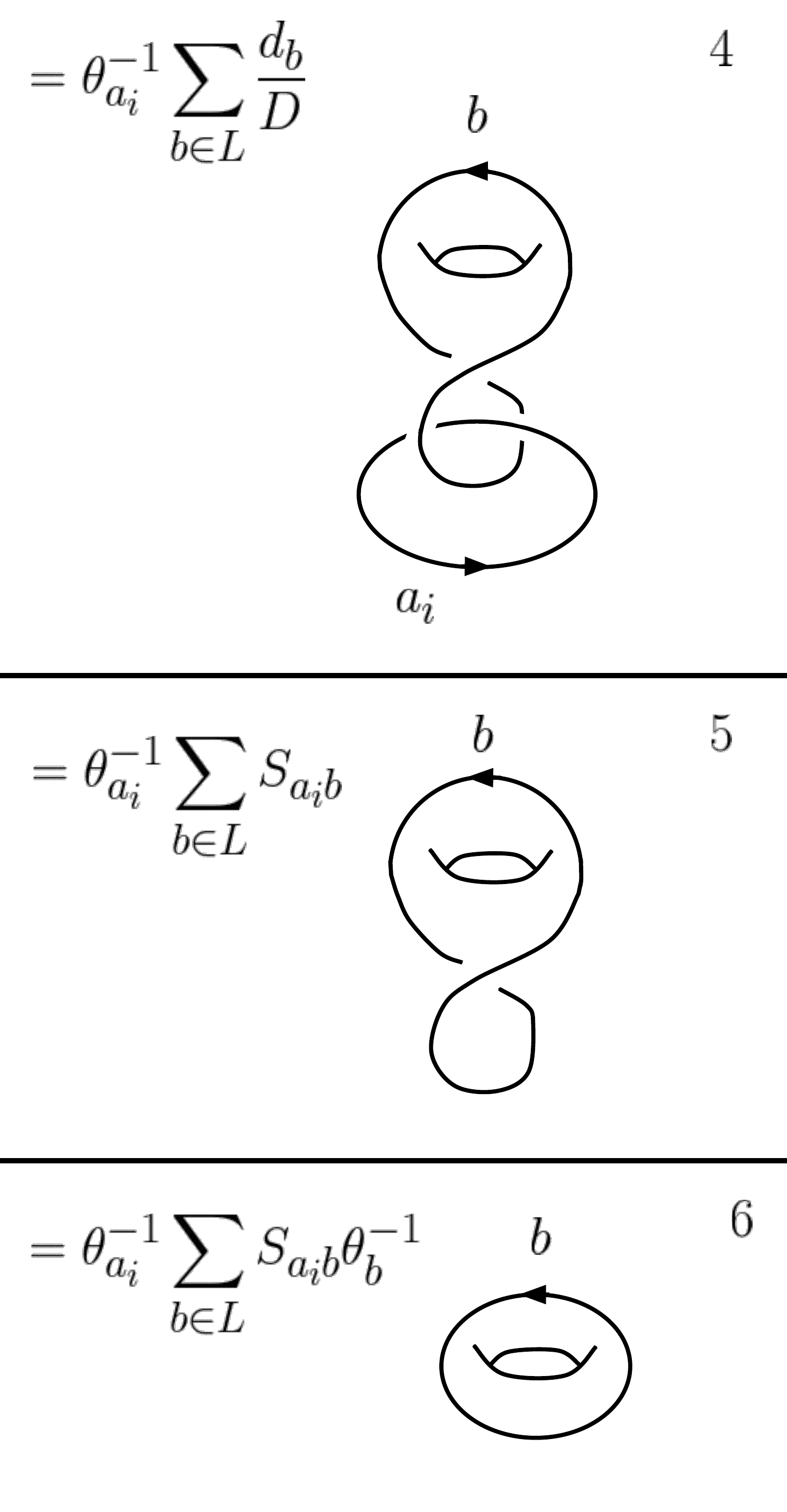}\]
\caption{Steps $4,5,$ and $6$}
\label{pSpec2}
\end{figure}
\clearpage
\clearpage
Then we can see that 
\[T_{2i}(\vec{a})=\theta_{a_i}^{-1}\sum_{b\in G} S_{a_i b}\theta_b^{-1} \vec{a}^\prime_{b}\]
Where $\vec{a}^\prime_b=(a_1,..,a_{i-1},b,a_{i+1},..,a_g)$ is determined from $\vec{a}$ by replacing the $i^th$ coordinate from $a_i$ to $b$.
Now define 
\[O:\mathcal{H}_G\rightarrow\mathcal{H}_G\]
defined by 
\[O(|a\rangle)=\sum_{b\in G}\theta_a^{-1}S_{a,b}\theta^{-1}_b|b\rangle.\]
Then we have that
\[T_{2i}(\vec{a})=O_i(\vec{a}).\]

\subsection{Clifford Operators}
\begin{theorem}
Let $\Sigma_g$ be the closed surface of genus $g$ and 
\[\rho_G:\mathrm{MCG}(\Sigma_g)\rightarrow PU(\mathcal{H}_G)\]
$\rho_G$ be the quantum representation coming from an abelian MTC with fusion rules determined by a finite abelian group $G$.  Then 
\[\rho_G(\mathrm{MCG}(\Sigma_g))\leq C_{g,G},\]
meaning the image of the mapping class group under this representation lies entirely in the $g^{th}$ Clifford group over $G$.  Moreover, each Humphries generator is sent to a Normalizer circuit over $\bigoplus_{i=1}^g G$.  
\end{theorem}
\begin{cor}
The computational framework using mapping class group representations arising from abelian anyon models can be classically efficiently simulated in at most polynomial time in the number of Dehn twists about Humphries generators of type $\gamma_{2i}$, the number of gates in the circuit, the number of cyclic factors of the group, and the logarithm of the orders of the cyclic factors. 
\end{cor}
\begin{proof}
As only normalizer gates can be achieved this follows from the theorem of Van den Nest, Theorem \ref{Vanden}.  Also note that the implementation of quantum Fourier transforms is obtained through the image of Dehn twists about Humphries generators of type $\gamma_{2i}$.
\end{proof}
\begin{lem}
$\sqrt{|G|}S_{x,y}$ is a bi-character, meaning
\[|\sqrt{|G|}S_{x,y}|=1\]
\[\sqrt{|G|}S_{x+y,g}=\sqrt{|G|}S_{x,g}\sqrt{|G|}S_{y,g}\]
\[\sqrt{|G|}S_{g,x+y}=\sqrt{|G|}S_{g,x}\sqrt{|G|}S_{g,y}\]
\end{lem}
\begin{proof}
This follows immediately as 
\[\sqrt{|G|}S_{x,y}=\exp(2\pi i b(x,y))\]
where $b(x,y)$ is bilinear.
\end{proof}
Now we return to the proof of our theorem.
\begin{proof}

We see that based on the structure computed above we need only show that $L$, $M$, and $O$ lie $C_{1,G}$, $C_{2,G}$, and $C_{1,G}$ respectively as tensoring with the identity operator will preserve that result and the root of unity $\frac{p_-}{D}$ can be ignored as these operators are only considered projectively.
\subsubsection{L}
Recall
\[L(|x\rangle)=\theta_x |x\rangle.\]
We first look to show that $L$ lies in $C_{1,G}$, and in particular that $L$ is a normalizer gate over $G$.  In fact we will show that $L$ is a quadratic phase gate, meaning $\theta_x$ is a quadratic phase.  We first note that $\theta_x$ is a root of unity, by Vafa's theorem, thus we need only show that 
\[\theta_{x+y}=\theta_x\theta_y B(x,y)\]
In fact we have 
\[\frac{\theta_{x+y}}{\theta_x\theta_y}=\sqrt{|G|}S_{x,y}\]
which as we have seen in Lemma $4.3$ is a bicharacter.  Thus $L$ is a quadratic phase gate and thus a normalizer gate.
\subsubsection{M}
Recall
\[M(|x\rangle\otimes |y\rangle)=\theta_{x-y}|x\rangle\otimes |y\rangle.\]
We we look to show that $M\in C_{2,G}$. In particular, we will show that $M$ is a normalizer circuit over $G\oplus G$.  We have 
\[M(|x+y\rangle)=\theta_{x-y}|x-y\rangle.\]
Then our computation will follow very similar to that of the one above.  We must show that 
\[\theta_{(x+a)-(y+b)}=\theta_{x-y}\theta_{a-b}B((x,y),(a,b)),\]
where $B((x,y),(a,b))$ is a bicharacter.
We have 
\[\theta_{(x+a)-(y+b)}=\theta_{(x-y)+(a-b)}=\theta_{x-y}\theta_{a-b}\sqrt{|G|}S_{x-y,a-b}.\]
We are left to show that 
\[B((x,y),(a,b))=\sqrt{|G|}S_{x-y,a-b}\]
is a bicharacter.  We have
\begin{align*}
&B((x+g,y+h),(a,b))=\sqrt{|G|}S_{(x+g)-(y+h),a-b}\\
& =\sqrt{|G|}S_{(x-y)+(g-h),a-b}=\sqrt{|G|}S_{x-y,a-b}\sqrt{|G|}S_{g-h,a-b}\\
&=B((x,y),(a,b))B((g,h),(a,b))\\
\end{align*} 
and an identical computation shows
\begin{align*}
&B((x,y),(a+g,b+h))=B((x,y),(a,b))B((x,y),(g,h)).\\
\end{align*}

Thus we have that $M\in C_{2,G}$ and in particular it is a normalizer gate over $G\oplus G$.
\subsubsection{O}
Recall
\[O(|x\rangle)=\sum_{y\in G}\theta_x^{-1}S_{x,y}\theta_y^{-1} |y\rangle\]
We look to show that $O\in C_{1,G}$ and in particular that $O$ is a normalizer gate over $G$.  We quickly see that we are pre-composing and post-composing with $\theta^{-1}_z$, which from above we have seen is a quadratic phase gate.  Thus we need only show that 
\[\sqrt{|G|}S:|x\rangle\mapsto \sum_{y\in G} S_{x,y}|y\rangle\]
is in is a normalizer gate over $G$.  Utilizing our lemma which proved that $S_{x,y}$ was a bicharacter we can in fact write 
\[\sqrt{|G|}S_{x,y}=\chi_{y}(x)\]
where $\chi_y$ is a character. Then we have 
\[S(|x\rangle)=\frac{1}{|G|}\sum_{y\in g}\chi_{y}(x)|y\rangle\]
which is exactly the global quantum Fourier transform.  Thus we have $S$ is a noramlizer gate over $G$ and so $O$ is as well.
\par  

Thus we have completed our proof of Theorem 1.  
\end{proof}
\section{General Anyons}

Though the $1$-qudit gates in our scheme always form a finite group, they are not always generalized Clifford gates as we show below for the Fibonacci anyon. 

\subsection{Fib}
The simple objects of Fib are $1$ and $\tau$. The only nontrivial fusion rule is
\[\tau\otimes \tau=1\oplus \tau.\] Let $\phi=\frac{1+\sqrt{5}}{2}$ be the golden ratio. Then we can write the evaluation moves explicitly as 
\[d_1=1,\qquad d_\tau=\phi\]
\[T=\left(\begin{array}{cc}
1 & 0 \\
0 & e^{4\pi i/5}\\

\end{array}\right)\]
\[S=\frac{1}{\sqrt{2+\phi}}\left(\begin{array}{cc}
1 & \phi \\
\phi & -1\\

\end{array}\right)\]
\[R^{\tau\tau}_1=e^{\frac{-4\pi i}{5}}\qquad R^{\tau\tau}_\tau=e^{\frac{3\pi i}{5}}\]
\[F^{\tau\tau\tau}_{\tau}=\left(\begin{array}{cc}
\phi^{-1} & \phi^{-1/2}\\ 
\phi^{-1/2} & -\phi^{-1}\\

\end{array}\right)\]

This case differs greatly from the previous case.  The most striking of these differences is the lack of a tensor product structure on $V_{Fib}(\Sigma)$.  We instead look at a computational subspace inside of $V(\Sigma)$.  In particular the subspace $(\mathbb{C}^2)^{\otimes g}$ restricting all of the $a_i$ labels to be $1$.  This leaves each genus to be encircled by either a $1$ or a $\tau$.  This computational subspace is even invariant under $T_0,T_1,$ and $T_{2i}$ for $i=1,...,g$.  Unfortunately this subspace is not invariant under $T_{2i+1}$ for $i=1,...,g-1$.  This lack of invariance does suggest that this computational subspace will inherently lead to leakage, but that does not rule this out as a promising model. 
\begin{theorem}
There does not exist a basis for $V(T^2)$ for which both $S$ and $T$ lie in the associated Clifford group on the single qubit. 
\end{theorem}
\begin{proof}
First we observe that $T^5=Id$.  Then as the order of the Clifford group is $24$ we know that as $5$ does not divide $24$ the only possibility is that in our chosen basis $T$ is the identity matrix.  So in our new ``normalized" basis we have 
\[T=\left(\begin{array}{cc}
1 & 0\\
0 & 1\\
\end{array}\right)\]
and
\[S=\frac{1}{\sqrt{2+\phi}}\left(\begin{array}{cc}
1 & e^{-4\pi i/5}\phi\\
\phi e^{4\pi i/5} & -1\\
\end{array}\right)\]
By explicit computation we can show $S$ is not a Clifford operator, even up to a global phase.  We quickly see that $S$ has order $2$.  Then we have $9$ matrices to compare this to, up to global phase.  Explicit computation (refer to Appendix A \ref{table}) shows that of these $9$ matrices, $4$ have the property that their off diagonals are equal, $3$ have the property that their off diagonals sum to zero, and the remaining two have at least one zero entry.  All three of these properties are preserved under global phases, but our matrix $S$ does not have these properties.  Thus in this computational basis $S$ is not a Clifford operator.  Then as this is the only basis that allowed $T$ to be a Clifford operator we have shown that it is not possible for both $S$ and $T$ to be Clifford operators in the same basis.
\end{proof}
\appendix

\section{The Abstract Clifford Group on One Qubit}
The results of this appendix are well known, but collected here for convenience.
\subsection{The Pauli Group on One Qubit}
We start by looking at the Pauli Group on one qubit.  This is a specialization of the definition given at the beginning of this paper.  In particular we have 
\[P_1:=\langle X,Y,Z\rangle=\{\pm Id, \pm iId, \pm X, \pm iX, \pm Y, \pm iY, \pm Z,\pm iZ\}\]
Abstractly this is a $16$ element group.  As we will only be working up to a global phase it is convenient for us to define
\[P:=\{\pm Id, \pm X, \pm Y, \pm Z\}.\]    Once a computational basis for the underlying $2$ dimensional Hilbert space is chosen, then we have a realization of this group as a matrix group.  Here we have 
\[X=\left(\begin{array}{cc}
0 & 1\\
1 & 0\\
\end{array}\right)\]
\[Y=\left(\begin{array}{cc}
0 & -i\\
i & 0\\
\end{array}\right)\]
\[Z=\left(\begin{array}{cc}
1 & 0\\
0 & -1\\
\end{array}\right)\]
\subsection{The Clifford Group on One Qubit}
Now the Clifford group on one qubit can be viewed as the normalizer of the Pauli group, up to overall global phases.
\begin{defn}
The Clifford group on one qubit is 
\[C_1:=\{U\in U(2):UpU^*\in P-\{\pm Id\}, p\in P-\{\pm Id\}\}/U(1)\]
\end{defn}
\begin{prop}
The Clifford group on one qubit has order $24$.
\end{prop}
\begin{proof}
We first note that conjugation must preserve the group structure, and in particular here we mean the multiplication of the Pauli matrices.  Thus as $Y=iXZ$, we will not need to specify the image of $Y$ under the conjugation.  Similarly $-X$ and $-Z$ will be determined by where $X$ and $Z$ are sent as well.  Thus we will only need to specify where $X$ and $Z$ end up.  We know that $X$ and $Z$ anti-commute and so  $UXU^*$ and $UZU^*$ will also need to anti-commute.  This tells us that $X$ can be sent to any element of $P-\{\pm Id\}$, but $Z$ can only be send to $P-\{\pm Id,UXU^*\}$.  Thus there are $6$ possibilities for $X$ to be sent to and $4$ possibilities for $Z$, and so $C_1$ has order $6\cdot 4=24$.
\end{proof}
\begin{theorem}\cite{G}
Similar to above, once a computational basis is chosen for the Hilbert space it is possible to describe $C_1$ explicitly.  In fact 
\[C_1=\langle H,Q\rangle\]
where 
\[H=\frac{1}{\sqrt{2}}\left(\begin{array}{cc}
1 & 1\\
1 & -1\\
\end{array}\right)\]
and
\[Q=\left(\begin{array}{cc}
1 & 0\\
0 & i\\
\end{array}\right)\]
\end{theorem}
We note that our description of $C_1$ is as a $24$ element group.  The usually order given to the group generated by $H$ and $Q$ would be $192$, but recall we have an equivalence up to global phase of the words in $H$ and $Q$.  In particular the factor of $8$ results in an overcounting seen from $(PQ)^3=e^{2\pi i/8}Id$ which for our purposes is the identity.
\begin{cor}
As $24$ element groups
\[C_1\cong S_4.\]
As a note, $S_4$ is the symmetry group of the cube.
\end{cor}
\begin{proof}
We see
\[S_4=\langle (1,2),(1,2,3,4)\rangle.\]
Then using the description afforded by Theorem $3$ we are done.
\end{proof}

We now provide a table of representatives of the elements of $C_1$ along with corresponding elements of $S_4$ coming from the isomorphism used in Corollary $1$.
\begin{center}\label{table}
{\renewcommand{\arraystretch}{1.35}%
 \begin{tabular}{||c c||||c c||} 
 \hline
 $C_1$ & $S_4$ & $C_1$ & $S_4$ \\ [0.5ex] 
 \hline\hline
 $\left(\begin{array}{cc}
 1 & 0\\
 0 & 1\\
 \end{array}\right)$ & $(1)$ & $\frac{1}{\sqrt{2}}\left(\begin{array}{cc}
 1 & i\\
 -1 & i\\
 \end{array}\right)$ & $(132)$  \\ 
 \hline
$\frac{1}{\sqrt{2}}\left(\begin{array}{cc}
 1 & 1\\
 1 & -1\\
 \end{array}\right)$ & $(12)$ & $\frac{1}{\sqrt{2}}\left(\begin{array}{cc}
 1 & -1\\
 -1 & -1\\
 \end{array}\right)$ & $(34)$ \\
 \hline
 $\left(\begin{array}{cc}
 1 & 0\\
 0 & i\\
 \end{array}\right)$ & $(1234)$ & $\left(\begin{array}{cc}
 0 & 1\\
 -1 & 0\\
 \end{array}\right)$ & $(12)(34)$\\
 \hline
 $\left(\begin{array}{cc}
 1 & 0\\
 0 & -1\\
 \end{array}\right)$ & $(13)(24)$ & $\left(\begin{array}{cc}
 0 & 1\\
 i & 0\\
 \end{array}\right)$ & $(24)$ \\
 \hline
 $\left(\begin{array}{cc}
 1 & 0\\
 0 & -i\\
 \end{array}\right)$ & $(1432)$ & $\left(\begin{array}{cc}
 0 & 1\\
 1 & 0\\
 \end{array}\right)$ & $(14)(23)$\\
 \hline
$\frac{1}{\sqrt{2}}\left(\begin{array}{cc}
 1 & 1\\
 i & -i\\
 \end{array}\right)$ & $(134)$ & $\frac{1}{\sqrt{2}}\left(\begin{array}{cc}
 1 & i\\
 -i & -1\\
 \end{array}\right)$ & $(14)$\\
  \hline
$\frac{1}{\sqrt{2}}\left(\begin{array}{cc}
 1 & 1\\
 -1 & 1\\
 \end{array}\right)$ & $(1423)$ & $\frac{1}{\sqrt{2}}\left(\begin{array}{cc}
 1 & -1\\
 -i & -i\\
 \end{array}\right)$ & $(123)$\\
  \hline
$\frac{1}{\sqrt{2}}\left(\begin{array}{cc}
 1 & 1\\
 -i & i\\
 \end{array}\right)$ & $(243)$ & $\frac{1}{\sqrt{2}}\left(\begin{array}{cc}
 1 & -i\\
-i & 1\\
 \end{array}\right)$ & $(1342)$\\
  \hline
$\frac{1}{\sqrt{2}}\left(\begin{array}{cc}
 1 & i\\
 1 & -i\\
 \end{array}\right)$ & $(234)$ & $\frac{1}{\sqrt{2}}\left(\begin{array}{cc}
 1 & -i\\
 -1 & -i\\
 \end{array}\right)$ & $(124)$\\
  \hline
$\frac{1}{\sqrt{2}}\left(\begin{array}{cc}
 1 & -1\\
 1 & 1\\
 \end{array}\right)$ & $(1324)$ & $\left(\begin{array}{cc}
 0 & 1\\
 -i & 0\\
 \end{array}\right)$ & $(24)$ \\  
  \hline
$\frac{1}{\sqrt{2}}\left(\begin{array}{cc}
 1 & -i\\
 1 & i\\
 \end{array}\right)$ & $(143)$ & $\frac{1}{\sqrt{2}}\left(\begin{array}{cc}
 1 & -i\\
 i & -1\\
 \end{array}\right)$ & $(23)$ \\
  \hline
$\frac{1}{\sqrt{2}}\left(\begin{array}{cc}
 1 & i\\
 i & 1\\
 \end{array}\right)$ & $(1243)$ & $\frac{1}{\sqrt{2}}\left(\begin{array}{cc}
 1 & -1\\
 i & i\\
 \end{array}\right)$ & $(142)$\\ [1ex] 
 \hline
\end{tabular}}
\end{center}

\end{document}